\documentclass[a4paper,twosided, 11pt]{article}
\usepackage[utf8]{inputenc} 
\usepackage[T1]{fontenc} 
\usepackage[a4paper]{geometry} 
\usepackage{amsmath, amssymb, mathtools, amsthm}
\usepackage{mathtools} 
\usepackage{mathrsfs}  
\usepackage{graphicx} 
\usepackage{xcolor}   
\usepackage[english]{babel} 
\usepackage{bbm}
\usepackage[title]{appendix}

\usepackage{caption}
\usepackage{subcaption}
\usepackage{comment}
\usepackage{geometry}
\usepackage{cite}
\usepackage{empheq}
\usepackage{float}
\usepackage{booktabs}
\usepackage{varwidth}
\usepackage{enumitem} 
\usepackage[colorlinks=true]{hyperref} 

\newcommand \commentout[1] {}

\newcommand{\R}{\mathbb{R}}

\newcommand{\Z}{\mathbb{Z}}

\newcommand {\Chi} {{\bf \raise 2pt \hbox{$\chi$}} }

\newcommand {\sgn} { {\rm sgn} }

\newcommand {\T}  { {\mathbb{T}} }

\newcommand {\p}   {\partial}

\newcommand*{\dd}{\mathop{\kern0pt\mathrm{d}}\!{}}

\newcommand*{\DD}{\mathop{\kern0pt\mathrm{D}}\!{}}

\newcommand*{\vv}{\mathbf{v}}
\newcommand*{\ww}{\mathbf{w}}

\theoremstyle{plain}
\newtheorem*{thm*}{Theorem}
\newtheorem{thm}{Theorem}[section]

\newtheorem{ass}[thm]{Assumption}
\newtheorem{proposition}[thm]{Proposition}

\theoremstyle{remark}
\newtheorem{remark}[thm]{\bf Remark}
\newtheorem{definition}[thm]{\bf Definition}

\newcommand{\beq}{\begin{equation}}
\newcommand{\eeq}{\end{equation}}
\newcommand{\bea} {\begin{array}{rl}}
\newcommand{\eea} {\end{array}}
\newcommand{\bepa}{\left\{ \begin{array}{l}}
\newcommand{\eepa} {\end{array}\right.}

\numberwithin{equation}{section}

\title{A cross-diffusion system with independent drifts and fast diffusion}

\author{
    Charles Elbar%
     \thanks{Universite Claude Bernard Lyon 1, ICJ UMR5208, CNRS, Ecole Centrale de Lyon, INSA Lyon, Université Jean Monnet, 69622
Villeurbanne, France. Email: elbar@math.univ-lyon1.fr} %
    \and 
    Filippo Santambrogio%
     \thanks{Universite Claude Bernard Lyon 1, ICJ UMR5208, CNRS, Ecole Centrale de Lyon, INSA Lyon, Université Jean Monnet, 69622
Villeurbanne, France. Email: santambrogio@math.univ-lyon1.fr}
}
\date{}

\begin{document}

\maketitle

\begin{abstract}
We study a one-dimensional cross-diffusion system for two interacting populations on the torus, with a fast-diffusion law with exponent $0< \alpha\le 1$  and different external potentials.  For arbitrary non-negative $L^{1}$ initial data with bounded entropy and a mixing condition we prove the existence of global weak solutions.  This extends the recent result of Mészáros, Parker from the linear diffusion ($\alpha=1$) to the fast-diffusion.   

\end{abstract}
\vskip .7cm

\noindent{\makebox[1in]\hrulefill}\newline

\section{Introduction}

Understanding how interacting species compete for space and resources is fundamental in many areas of biology, from cell migration and tumor growth to large-scale population dynamics. A common feature in these systems is that individuals adjust their movement to avoid overcrowded regions while also responding to heterogeneities in their environment.

In this paper we study a system that takes into account these dynamics for two interacting populations, whose densities are denoted by $\rho(t,x)$ and $\mu(t,x)$. We consider their evolution on a one-dimensional torus $\T := \R/\Z$ over a time interval $[0,T]$ for some $T>0$ (if the choice of the torus is motiviated by simplicity in order to avoid boundary issues, the choice of the one-dimensional setting is, on the other hand, crucial for our approach, as the reader will see). They are given by the following cross-diffusion system:

\begin{align}
\partial_t\rho - \partial_x\left(\rho\,\partial_x f'(\rho+\mu)\right)
              - \partial_x\left(\rho\,\partial_x V\right) &= 0,
              \label{eq:rho}\\
\partial_t\mu  - \partial_x\left(\mu\,\partial_x f'(\rho+\mu)\right)
              - \partial_x\left(\mu\,\partial_x W\right) &= 0,
              \label{eq:mu}
\end{align}

Here, $\rho$ and $\mu$ have initial densities $\rho(0,\cdot) = \rho_0$ and $\mu(0,\cdot) = \mu_0$ assumed to be non-negative. The functions $V,W\in C^3(\T)$ are external potentials, and the nonlinear function $f':\R\to\R$ is the pressure, assumed to be of fast-diffusion type. We choose to call it $f'$ since it is its antiderivative $f$
 which appears when expressing this system as a gradient flow. The main goal of this paper is to prove the existence of weak solutions to System~\eqref{eq:rho}--\eqref{eq:mu}.\\

\textbf{Biological relevancy.} From a biological point of view, System~\eqref{eq:rho}--\eqref{eq:mu} models several important mechanisms that we describe here. The term \(\partial_x\bigl(\rho\,\partial_x f'(\rho+\mu)\bigr)\) (and similarly for \(\mu\)) is a \emph{cross-diffusion} effect: the movement of each population is influenced by the gradient of a pressure that depends on the total density \(\rho+\mu\) as in~\cite{MR4072681,David03042023}. This models the fact that the two populations compete for space and tend to avoid overcrowded regions. Such cross-diffusion mechanisms have been used to describe many biological effects, such as the territorial segregation in ecological systems or the movement of cells in tissues.

In addition to this pressure-driven movement, the populations feel external potentials, $V(x)$ and $W(x)$, which model directed drift. They represent possible spatial heterogeneities in the environment, such as gradients of nutrients, chemotactic signals, or the presence of rivers, water columns etc.~\cite{MR1654407,MR1372792,MR2290095,MR1961241,MR2462700,MR3017035,MR2501474}. An important aspect of the model is that it allows these potentials to be different for each species ($V \neq W$), which allows the biological reality that different species or cells types may have different environmental preferences or may respond differently to chemotactic signals.

Here we choose a \emph{fast diffusion} type nonlinearity for \(f\), as our proof does not cover other regimes. More precisely, $f(s)$ is assumed to behave like \(s^\alpha\) for some \(0<\alpha\leq 1\), meaning that the diffusivity increases when the total density decreases. Indeed, we will see that the sum $S:=\rho+\mu$ of the two densities will be driven by a diffusion PDE with a diffusive term $\Delta(\Phi(S))$, where $\Phi'(s)=sf''(s)$; the function $\Phi$ is homogeneous to $f$ and it also behaves as $s^\alpha$, so that we face a fast-diffusion equation for $S$. This diffusivity is biologically relevant and it models the observation that, at low densities, populations tend to spread faster to occupy available space and to colonize low-density regions quickly. For example, in ecology, animal populations or plant seeds placed in an empty space often expand their range aggressively when the competition is low, to have access to resources before others arrive. In tissue dynamics, motile cells at the front and close to the wound move faster towards the wound~\cite{Nikolic2006}. \\

\textbf{Comparison with the litterature.} Cross‐diffusion systems have attracted a lot of interest, see for instance~\cite{MR3350617,MR4497322,MR1616969,MR2083864,MR2220063,2024arXiv241205751G,MR4664467,MR4629864} because they allow one species’ movement to depend on another’s density gradient. A similar system than~\eqref{eq:rho}--\eqref{eq:mu}, with reaction terms instead of different potentials, was studied in~\cite{MR736508,MR2652018,MR4664467}  and then with optimal transport methods in~\cite{MR3870087} in the one-dimensional case. Indeed the system is the gradient flow of the functional 

$$
F[\rho,\mu] = \int_{\T}f(\rho+\mu) + \rho V + \mu W
$$

with respect to the Wasserstein metric.  The very first existence result for weak solutions of our system appeared in~\cite{MR3795211}, under very restrictive assumptions the potentials and initial density, mainly that the density
with stronger drift in $x$ direction sits on the right side on the $x$-axis, which corresponds to segregated solutions. This limits both generality and biological applications.  A significant step came in~\cite{2025arXiv250418484M}, where the authors tackled the one‑dimensional case with \(\Phi(S)=S\) (\(\alpha=1\)). They proved existence for broad classes of initial data and interaction potentials.  Unfortunately, their proof is rigid: any small perturbation of the logarithm breaks the key identities and causes the argument to fail.

In this paper, we improve on this result.  We show that as long as we choose $\Phi(s)=s^\alpha$ for $0 < \alpha\le1,$
one can obtain weak solutions under the same general assumptions on data and potentials.  Our method does not depend on the exact form of $\Phi$ and $f$, and we would be able to cover in fact more general diffusion laws but for simplicity of the exposition we will stick to the power case.\\

\textbf{Notations and functional settings.} We denote by $L^{p}(\T)$, $H^{s}(\T)=W^{s,2}(\T)$ the usual Lebesgue and Sobolev spaces, and by $\|\cdot\|_{L^{p}}$, $\|\cdot\|_{H^s}$ their corresponding norms. We often write $C$ for a generic constant appearing in the different inequalities. Its value can change from one line to another, and its dependence to other constants can be specified by writing $C(a,...)$ if it depends on the parameter $a$ and other parameters.

\subsection{Main result}

We state our main result here, that is the existence of weak solutions of System~\eqref{eq:rho}--\eqref{eq:mu}.  We require $f$ to satisfy a fast--diffusion type assumption. 

\begin{ass}\label{ass:nonlinearity}
We assume that $\Phi(s)=s^{\alpha}$ for some $0< \alpha\le 1 $, where $\Phi$ is defined via $\Phi'(s)=sf''(s)$.
\end{ass}

Then we define the notion of weak solutions:

\begin{definition}[Weak solution]\label{def}
A pair $(\rho,\mu)$ is a \emph{weak solution} of
\eqref{eq:rho}--\eqref{eq:mu} on $(0,T)\times\T$ if
\begin{enumerate}
\item $\rho,\mu\ge0$ and
      $\rho,\mu\in L^\infty((0,T);L^1(\T))$;
\item $\sqrt{\rho}\partial_{x}f'(\rho+\mu),\,  \sqrt{\mu}\partial_x f'(\rho+\mu)\in L^{2}((0,T); L^{2}(\T))$;
\item for every $\varphi,\psi\in C_c^\infty([0,T)\times\T)$,
      \begin{align*}
        &\int_{0}^{T}\int_{\T}
        \left[-\rho\,\partial_t\varphi
              +\rho\,\partial_x f'(\rho+\mu)\,\partial_x\varphi
              +\rho\,\partial_xV\,\partial_x\varphi\right]dx\,dt
          =\int_{\T}\rho_0(x)\varphi(0,x)\,dx,\\
        &\int_{0}^{T}\int_{\T}
        \left[-\mu\,\partial_t\psi
              +\mu\,\partial_x f'(\rho+\mu)\,\partial_x\psi
              +\mu\,\partial_xW\,\partial_x\psi\right]dx\,dt
          =\int_{\T}\mu_0(x)\psi(0,x)\,dx.
      \end{align*}
\end{enumerate}
\end{definition}

Our main result reads: 

\begin{thm}[Global existence]\label{thm:main}
Suppose Assumption~\ref{ass:nonlinearity} holds and take $V,W\in C^3(\T)$, $\rho_{0},\mu_{0}\in L^{1}(\T)$ two non-negative densities such that $\rho_0\log\rho_0, \mu_0\log\mu_0\in L^{1}(\T^d)$ and $\log(\rho_0/\mu_0)\in BV(\T)$.
Then there exists at least one weak solution
$(\rho,\mu)$ of \eqref{eq:rho}--\eqref{eq:mu} on $[0,T]\times\T$
in the sense of Definition~\ref{def}. 
\end{thm}

\begin{remark}
\begin{itemize}
\item The method we use to prove Theorem~\ref{thm:main} covers also more general diffusions laws than Assumption~\ref{ass:nonlinearity}. In fact, we mainly use the integrability of some functionals involving $\Phi$. For instance both taking $\Phi(s)= s^2 + s^{-1}$ work, but we are not able to cover the case $\Phi(s) = s^2$.

\item In principle, the method could be adapted to most of nonlinearities, including the porous medium $\Phi(s)=t^2$, as long as one can prove lower and upper bounds on the total density $S=\rho+\mu$, up to considering also suitable initial data. However, we were not able to derive such estimates, as it does not seem that comparison principle holds.
\end{itemize}
\end{remark}

Also, to simplify the exposition, we only prove Theorem~\ref{thm:main} assuming we deal with smooth solutions and we obtain a priori estimates. All the computations can be justified in an approximated scheme, exactly as in~\cite{2025arXiv250418484M}.

\subsection{Sketch of proof}

We outline here the core of the argument.  Let us assume that we have an approximating scheme $\{\rho_n\}_n$, $\{\mu_n\}_n$, for instance adding an artificial viscotity. The difficulty in obtaining weak solution of the system is to pass to the limit in the nonlinear fluxes.

Consider \eqref{eq:rho}.  The critical product is
$$
\rho_n\,\partial_x f'(\rho_n+\mu_n),
$$
because both factors are known a priori to converge only weakly. Indeed:

\begin{itemize}
\item    Since both equations preserve the mass, we easily get 
      $L^1$‐bounds on $S_n:=\rho_n+\mu_n$.  Therefore
      $\rho_n,\mu_n\in L^\infty((0,T);L^1(\T))$, hence $\sqrt{\rho_n}$ and $\sqrt{\mu_n}$  converge
      weakly (up to a subsequence)  in $L^2$.
\item  Looking at the dissipation of the energy $F$ we easily obtain
      $\sqrt{\rho_n}\partial_x f'(\rho_n+\mu_n), \sqrt{\mu_n}\partial_x f'(\rho_n+\mu_n)\in L^2\bigl((0,T)\times\T\bigr)$, so this factor
      also converges weakly in $L^2$.
\end{itemize}

To identify the limit of their product we therefore need to prove strong convergence
 of $\rho_n$ and $\mu_n$. In the fast-diffusion case, we are able to obtain an $H^1$ bound on a quantity involving the sum $S$. This alone is not enough and we need a second gradient estimate involving another independent scalar of $(\rho_n,\mu_n)$.

A common approach in such cross-diffusion systems is to consider
$r_n=g(\rho_n/\mu_n)$.  For $\Phi(s)=s$ (or equivalently $f'(s)=\log(s))$ the authors of \cite{2025arXiv250418484M} introduce
$$
r_n:=\log\Bigl(\frac{\rho_n}{\mu_n}\Bigr)+\bigl(V-W\bigr),
$$
and obtain the transport equation
$$
\partial_t r_n  =  (\partial_x r_n)v_n  + \text{l.o.t.},
$$
where $v_n$ is a velocity field (with no special regularity required) and l.o.t. are terms
with integrable spatial derivatives. It is well-known that in one space dimension transport equations preserve the BV norm. Here, taking $\partial_x$, then
multiplying by $\operatorname{sign}(\partial_x r_n)$ and integrating over $\T$
gives the $\mathrm{BV}$‐estimate
$$
\int_\T\lvert\partial_x r_n(t,x)\rvert\,dx  \le  C\quad
\forall\,t\in[0,T].
$$
Coupled with the gradient bound on $S_n$, this is enough to deduce strong
convergence of $\rho_n$ and $\mu_n$. In our settings, we modify the idea and rather than seeking for a transport equation satisfied by a ratio variable we work directly
with space derivatives. Set
$$
r_n=\log\left(\frac{\rho_n}{\mu_n}\right),\quad
u_n=\partial_x r_n-(\partial_x V(x) - \partial_x W(x))\, y(S_n),
$$
where
$y$ is choosen carefully in terms of $\Phi$.  We obtain that $u_n$ remains bounded in $L^1$.  Controlling $\partial_x r_n$ reduces to bounding
$\|y(S_n)\|_{L^1}$. In the fast diffusion case, we obtain $y(s)=Cs^{1-\alpha}$ and hence $y(S_n)$ is integrable. Then we combine the $H^1$-bound on $S_n$, and the $\mathrm{BV}$‐bound on $r_n$ to obtain $\mathrm{BV}$ bounds on $\rho_n$ and $\mu_n$. Some difficulties arise because these $\mathrm{BV}$ bounds are not integrable in time, which does not occur when $\alpha=1$. But adapating the proof of the Aubin–Lions lemma yields strong convergence of $\rho_n$ and $\mu_n$ in $L^1$ towards some $\rho,\mu$.
With this we are able to show that the product $\rho_n\,\partial_x f'(\rho_n+\mu_n)$ converges to the correct limit, and the pair $(\rho,\mu)$ is a weak solution of \eqref{eq:rho}--\eqref{eq:mu}. All the subsequent a priori estimates are independent of $n$ and therefore we drop the index for clarity.

\section{Estimates on $S$}

We consier in this section the behavior of the sum $S$ of the two densities, defined as $S:=\rho+\mu$.
We begin by adding \eqref{eq:rho} and \eqref{eq:mu}.  We obtain
$$
\partial_t S  = \partial_t(\rho+\mu)
 = \partial_x \left(\rho\partial_x f'(\rho+\mu) \right)
 + \partial_x \left(\mu\partial_x f'(\rho+\mu) \right)
 + \partial_x \left(\rho\partial_x V \right)
 + \partial_x \left(\mu\partial_x W \right).
$$
Since $\rho\partial_x f'(S)+\mu\partial_x f'(S)  = \partial_x \Phi(S)$, we get
\begin{equation}\label{eq:prelim_dS}
\partial_t S 
 =  
\partial_{xx} \Phi(S)
 + \partial_x \left(\rho\partial_x V + \mu\partial_x W \right).
\end{equation}

This expression yield estimates on $S$.

\begin{proposition}[Spatial estimate on $S$]\label{prop:spatial_S}
We have the following estimates on $S$:
\begin{align*}
&S\in L^{\infty}((0,T); L^{1}(\T)), \quad \partial_x S^{\alpha-\frac{1}{2}}\in L^{2}((0,T)\times \T),\quad S^{\alpha-\frac{1}{2}}\in L^{2}((0,T); L^{\infty}(\T)),\\
&\partial_x \log S\in L^{2}((0,T)\times \T), \quad \partial_x S^{\frac{\alpha}{2}}\in L^{2}((0,T)\times \T) \quad S^{\frac{\alpha}{2}} \in L^{2}((0,T); L^{\infty}(\T)),\\
&S\in L^{2-\alpha}((0,T)\times \T),  \quad \partial_x S^{1-\alpha}\in L^{1}((0,T)\times \T), \quad  S^{1-\alpha}\in L^{1}((0,T); L^{\infty}(\T)). 
\end{align*}
\end{proposition}

\begin{proof}
We first recall that the nonnegativity of both $\rho$ and $\mu$ can be obtained from the transport equation structure and the nonnegativity of the initial conditions. The first estimate is obtained after integrating~\eqref{eq:prelim_dS} in space and using periodic boundary conditions as well as the integrability of the initial conditions. 
We recall that $\Phi(S)= S^{\alpha}$ for $0< \alpha\le 1$. We only do the proof of the second estimate for $\alpha<1$, the proof for $\alpha=1$ being similar by considering the dissipation of $\frac{d}{dt}\int_{\T}S\log S$.
 Using~\eqref{eq:prelim_dS}, $\rho,\mu\le S$,  $|\partial_x V|, |\partial_x W|\le C$ and Young's inequality we obtain for $0\le \beta\le  1$

\begin{multline*}
\frac{d}{dt}\int_{\T}-S^{\beta}  + \alpha\beta(1-\beta)\int_{\T}S^{-3+\alpha+\beta}(\partial_x S)^2 \\= \beta(\beta-1)\int_{\T}\rho\partial_x V S^{\beta-2}\partial_xS+ \beta(\beta-1)\int_{\T}\mu\partial_x W S^{\beta-2}\partial_xS\\
\le \frac{\alpha\beta(1-\beta)}{2}\int_{\T}S^{-3+\alpha+\beta}(\partial_x S)^2 + C \int_{\T}S^{\beta+1-\alpha}.
\end{multline*}

Therefore we have

\begin{equation*}
\frac{d}{dt}\int_{\T}-S^{\beta} + \frac{2\alpha\beta(1-\beta)}{(\alpha+\beta-1)^2}\int_{\T}\left(\partial_x S^{\frac{\alpha+\beta-1}{2}}\right)^2 \le  C \int_{\T}S^{\beta+1-\alpha}.
\end{equation*}
Note that in this expression, the term in the middle reads $\frac{\alpha\beta(1-\beta)}{2}\int_\T \left(\partial_x \log S\right)^2$ when $\alpha+\beta=1$.

As $S\in L^{\infty}((0,T); L^{1}(\T))$, to obtain a priori estimates from this dissipation we first choose $\beta=\alpha$  and we obtain the second estimate after integrating in time and using also the integrability of the initial condition. The third estimate is a consequence of the first two estimates. Choosing then $\beta=1-\alpha$ and using the seventh estimate that is proved independently below we obtain the fourth estimate (indeed, for the validity of the fourth estimate we would need $S\in L^{2-2\alpha}$ and the seventh provides $S\in L^{2-\alpha}$; of course, when $\alpha>1/2$ this argument is not necessary).
To obtain the fifth estimate we compute the sum of the dissipation of each entropies: 
\begin{equation*}
\frac{d}{dt}\int_{\T}\rho\log\rho +\mu\log\mu + \int_{\T}f''(S)|\partial_xS|^2 = \int_{\T} \rho \partial_{xx}V +\int_{\T}\mu\partial_{xx}W\le C.   
\end{equation*}
The last estimate of this inequality is a consequence of mass conservation and the assumptions on $V,W$. Since $f''(S)|\partial_x S|^2 = C|\partial_x S^{\frac{\alpha}{2}}|^2$ we obtain the fifth estimate. We deduce the sixth estimate from the fifth.
From the eighth estimate we obtain the ninth, so it only remains to prove the seventh and eighth estimates. By using the previous estimates and Jensen's inequality we obtain:

\begin{align*}
\int_{0}^{T}\int_{\T}S^{2-\alpha}&= \int_{0}^{T}\int_{\T}SS^{1-\alpha}
\lesssim  \int_{0}^{T}\|S^{1-\alpha}\|_{L^{\infty}} \lesssim \int_{0}^{T}\int_{\T}\left|\partial_x S^{1-\alpha}\right|\lesssim \int_{0}^{T}\int_{\T}S^{1-\frac{3\alpha}{2}}|\partial_x S^{\frac{\alpha}{2}}|\\
&\lesssim \sqrt{\int_{0}^{T}\int_{\T}S^{2-3\alpha}}\sqrt{\int_{0}^{T}\int_{\T}|\partial_x S^{\frac{\alpha}{2}}|^2} \lesssim \left(\int_{0}^{T}\int_{\T}S^{2-\alpha}\right)^{\frac{1}{2}\frac{2-3\alpha}{2-\alpha}}
\end{align*}
From these inequalities we obtain the last two estimates: the seventh because $\frac{1}{2}\frac{2-3\alpha}{2-\alpha}<1$ and then the eighth one using the seventh estimate. 
\end{proof}

\begin{remark}
Eventhough such an estimate is not necessary in our analysis, we underline that it would be possible to obtain $S\in L^{1+\alpha}((0,T)\times \T)$ by computing the dissipation of the $H^{-1}(\T)$ norm of $S$.
\end{remark}

\section{The equations on $S$ and $r$}
We recall that we use the notations
$$
S(t,x) := \rho(t,x) + \mu(t,x), 
\qquad
r(t,x) := \log \left(\rho(t,x) \right) - \log \left(\mu(t,x) \right).
$$
We have
\begin{equation}\label{eq:rho_mu_in_terms_of_S_r}
\rho  = \frac{Se^r}{e^r+1},
\qquad
\mu  = \frac{S}{e^r+1}.
\end{equation}
Observe also that we have
$$
\rho-\mu = Sh(r), 
$$
where we define
\begin{equation}\label{eq:def_h}
h(r) := \frac{e^r-1}{e^r+1}. 
\end{equation}
with $h'(r)  
 = \frac{2e^r}{(e^r+1)^2}$. We also introduce
\begin{equation}\label{eq:def_g}
g(r) := \log \left(\frac{e^r}{(e^r+1)^2} \right) 
\end{equation}
so that $g'(r)  = \frac{1-e^r}{1+e^r}$.

Next we rewrite the term $\rho\partial_x V + \mu\partial_x W$ in a convenient form.  Set
\begin{equation}\label{eq:def_alpha_beta}
\vv  := \frac{\partial_x V+ \partial_x W}{2},
\qquad
\ww  := \frac{\partial_x V- \partial_x W}{2}.
\end{equation}
Then we obtain
$$
\rho\partial_x V + \mu\partial_x W
 =   \left(\rho+\mu \right) \vv
 +  \left(\rho-\mu \right) \ww. 
$$
Hence, using \eqref{eq:rho_mu_in_terms_of_S_r} and \eqref{eq:def_h},
\begin{equation}\label{eq:rhoV_plus_muW}
\rho\partial_x V + \mu\partial_x W
 =  S\vv  +  Sh(r)\ww.
\end{equation}
Plugging \eqref{eq:rhoV_plus_muW} into \eqref{eq:prelim_dS} gives
$$
\partial_t S 
 =  \partial_{xx} \Phi(S)
 + \partial_x \left(S\vv + Sh(r)\ww \right).
$$
We now expand the last term carefully:

$$
\partial_x \left(S\vv + Sh(r)\ww \right)
= \partial_x S\vv  + S\partial_x\vv
 + \partial_x S h(r)\ww + Sh'(r)\partial_x r\ww
 + Sh(r)\partial_x\ww.
$$
Hence
\begin{equation}\label{eq:dS_final}
\partial_t S
=\partial_{xx} \Phi(S)
 + \partial_x S \left[\vv + h(r)\ww \right]
 + S \left[\partial_x\vv + h(r)\partial_x\ww \right]
 + S\ww h'(r)\partial_x r.
\end{equation}
\bigskip

Next we derive the evolution equation for $r$. We obtain 
$$
\partial_t r 
= \frac{1}{\rho} \left[\partial_x(\rho\partial_x f'(S)) + \partial_x(\rho\partial_x V) \right]
 - \frac{1}{\mu} \left[\partial_x(\mu\partial_x f'(S)) + \partial_x(\mu\partial_x W) \right].
$$

Notice that the terms $\partial_{xx} f(S)$ cancel out, thus:

$$
\partial_t r 
= \partial_x r \partial_x f'(S)
 +  \partial_x \log\rho\, \partial_x V - \partial_x \log\mu\, \partial_x W 
 + \partial_{xx} V  - \partial_{xx} W.
$$

We now handle the mixed term 

\begin{align*}
\partial_x \log \rho \partial_x V - \partial_x \mu \partial_x W &= \left(\partial_x \log \rho + \partial_x \log\mu \right)\ww + \left(\partial_x \log \rho - \partial_x \log\mu \right)\vv\\
& =2\partial_x\log S\ww+ \partial_x r\left(g'(r)\ww +\vv\right) .
\end{align*}

Putting everything together, we obtain
\begin{equation}\label{eq:dtr}
\partial_t r = \partial_x r \left(\partial_x f'(S)+g'(r)\ww + \vv\right) + 2\ww\partial_x\log S + \partial_{xx}V-\partial_{xx}W.
\end{equation}

\section{BV estimates and compactness}

In this section we prove $\mathrm{BV}$ estimates on $r$ and show that they imply $\mathrm{BV}$ estimates on $\rho$ and $\mu$. Combined with estimates on $\partial_t\rho$ and $\partial_t\mu$ and adapting the proof of the Aubin-Lions lemma, we are able to prove strong convergence of the approximating scheme towards a weak solution of~\eqref{eq:rho}--\eqref{eq:mu}. We start with $\mathrm{BV}$ estimates on $r$. 

\subsection{BV estimates on $r$}

\begin{proposition}
Under the assumptions of Theorem~\ref{thm:main} there exists $C$ such that: 

$$
\sup_{t\in[0,T]}\int_{\T}|\partial_xr(t,\cdot)|\le C. 
$$
\end{proposition}

\begin{proof}
In order to derive a bound in $\mathrm{BV}$ for $r$ we differentiate the equation for $r$ with respect to $x$ and obtain

\begin{equation}\label{eq:w_full}
\partial_{t}\partial_{x}r = \partial_{x}\left(\partial_x r \left(\partial_x f'(S)+g'(r)\ww + \vv\right)\right)    + 2\partial_{xx}\log S \ww + 2\partial_x\ww\partial_{x}\log S + \partial_{xxx}V-\partial_{xxx}W.
\end{equation}

We now introduce a carefully chosen shift that cancels the highest‐order terms.  We fix a smooth function $\gamma=\gamma(x) $ and another function $y=y(S)$ depending only on $S$ to be choosen later.  Define
\begin{equation}\label{eq:def_u}
u(t,x) := \partial_{x}r(t,x)  - \gamma \,y \left(S(t,x) \right).
\end{equation}

We first compute 
$$
\partial_t u 
= \partial_t \partial_x r 
 - \gamma\,y' \left(S \right)\partial_t S.
$$

Also recall \eqref{eq:w_full} for $\partial_t \partial_x r$.  Subtracting $\gamma\,y'(S)\partial_t S$ from \eqref{eq:w_full} and using \eqref{eq:def_u}, we get

\begin{align*}
\partial_{t}u 
&= \partial_{x}\left(u\left(\partial_{x}f'(S)  +  g'(r)\ww  +  \vv\right)\right) 
   + \partial_{x}\left(\gamma y(S)\left(\partial_{x}f'(S)  +  g'(r)\ww  +  \vv\right)\right) \\
&\quad - \gamma y'(S)\partial_{t}S 
   +  2\frac{\partial_{xx}S}{S}\ww
   -  2\left(\partial_x \log S\right)^2\ww
   +  2\partial_{x}\log S\,\partial_x\ww
   +  \partial_{xxx}V  - \partial_{xxx}W,
\end{align*}

Then we use the expression for the time-derivative of $S$:
\begin{multline*}
\partial_t S = \Phi'(S)\partial_{xx}S +\Phi''(S))|\partial_x S|^2 + \partial_x S(\vv + h(r)\ww) \\+ S(\partial_x\vv + \partial_x\ww h(r)) + S\ww h'(r)u + S\ww h'(r)\gamma  y(S)
\end{multline*}

Thus

\begin{equation}\label{eq:dtu}
\begin{split}
\partial_{t}u 
&= \partial_{x}\left(u\left(\partial_{x}f'(S) + g'(r)\ww + \vv\right)\right) 
   -  \gamma y'(S)S\ww h'(r)u \\
&\quad + \gamma y(S)f''(S)\partial_{xx}S 
   - \gamma y'(S)\Phi'(S)\partial_{xx}S 
   + 2\frac{\partial_{xx}S}{S}\ww \\
&\quad + \partial_{x}\left(\gamma y(S)f''(S)\right)\partial_{x}S 
   + \partial_{x}\left(\gamma y(S)\left(g'(r)\ww+\vv\right)\right) \\
&\quad - \gamma y'(S)\left[\Phi''(S)|\partial_{x}S|^2 
   + \partial_{x}S\left(\vv + h(r)\ww\right) 
   +  S\left(\partial_x\vv + \partial_x\ww h(r)\right)+ \gamma y(S)S\ww h'(r)\right] \\
  &\quad -  2\left(\partial_x\log S\right)^2\ww 
   +  2\partial_{x}\log S\partial_x\ww 
   +  2\partial_{xx}\ww.
\end{split}
\end{equation}

We choose $\gamma =2\ww$ and recalling $f''(S)=\frac{\Phi'(S)}{S}$ we choose $y$ so that it satisfies the condition :
\begin{equation}
y(s)\frac{\Phi'(s)}{s}  - y'(s)\Phi'(s)  + \frac{1}{s}  = 0.
\end{equation}
to cancel the second line with the second derivatives of $S$ involved. 
We then take
$$
y(s)= -s\int_s^{\infty}\frac{1}{\sigma^2\Phi'(\sigma)}d\sigma,
$$
where the choice integrate on $(s,\infty)$ instead of $(0,s)$ is justified by the behavior of the integrand in the case $\Phi(s)=s^{\alpha}$. In this case, we obtain $y(s)=-\frac{s^{1-\alpha}}{\alpha^2}$. Note that in the case $\alpha=1$ this means $y(s)=-1$ and the quantity $u$ becomes exactly $\partial_x(r+V-W)$, which means that we come back to the very same computation as in \cite{2025arXiv250418484M}. Now we multiply Equation~\eqref{eq:dtu} by $\sgn(u)$ and integrate in space. The transport term involving $u$ disappears by periodic boundary conditions. As we chose $\gamma $ and $y(S)$  such that the second line disappears, we are only left with lower order terms. We show that there exists $C(t)\in L^{1}(0,T)$ such that we have
$$
\frac{d}{dt}\int_{\T}|u|\le C(t) \left(1 + \int_{\T}|u|\right).
$$

All the functions in $x$, that is $\gamma,\vv$ and $\ww$ are bounded together with their derivatives; also $h, h', g'$ are uniformly bounded.  Using the definition of $\Phi$ and $y$ we first look at all the terms including $|\partial_xS|^2$ and we see that they are all of the form $c\gamma |\partial_xS|^2$ since all functions $y'f''$, $yf'''$ and $y'\Phi''$ are proportional to $s^{-2}$. This allows to obtain:

\begin{equation}
\begin{split}
\frac{d}{dt}\int_{\T}|u| 
&\le C\int_{\T}|y'(S)S||u|\\
&+ C\int_{\T}\frac{||\partial_x S|^2}{S^2} \\
& + C\int_{\T}\left(\left|\frac{y(S)\Phi'(S)}{S}\right|+ |y'(S)| + \frac{1}{S}\right)|\partial_x S| +\\ &+ C\int_{\T}|y(S)|  + S|y'(S)| + |y'(S)y(S)S|.
\end{split}
\end{equation}

We control the right-hand in the following with Proposition~\ref{prop:spatial_S}:
\begin{enumerate}
\item $\int_{\T}|y'(S) S| |u|$:  $y'(S)S = C S^{1-\alpha}\in L^{1}((0,T); L^{\infty}(\T))$;
\item $\int_{\T}\frac{|\partial_x S|^2}{S^2}$: $\partial_x \log S\in L^{2}((0,T)\times \T)$; 
\item $\int_{\T}\left|\frac{y(S)\Phi'(S)}{S}\right||\partial_x S|$: $\frac{y(S)\Phi'(S)}{S}= C S^{-1}$ and $\partial_x \log S \in L^{2}((0,T)\times\T)$ 
\item $\int_{\T} |y'(S)||\partial_x S|$: $y'(S)\partial_x S = C \partial_x  S^{1-\alpha}\in L^{1}((0,T)\times \T)$; 
\item $\int_{\T}|y(S)|$: $y(S)=CS^{1-\alpha}\in L^{1}((0,T); L^{\infty}(\T))$ (but, using $1-\alpha<1$, we also have $S^{1-\alpha}\in L^{\infty}((0,T); L^{\infty}(\T))$; 
\item $\int_{\T}|y'(S)S|$: we notice that $y'(S)S=C y(S)$ so the computation is similar to the 5th one;
\item $\int_{\T}|y'(S)Sy(S)|$: $y'(S)Sy(S)= C y(S)^2= CS^{2-2\alpha}$ which is bounded in $L^{1}((0,T); L^1(\T))$ (since $S\in L^{2-\alpha}((0,T)\times \T)$).
\end{enumerate}

From 
$$
\frac{d}{dt}\int_{\T}|u|\le C(t) \left(1 + \int_{\T}|u|\right)
$$
and Gronwall's lemma we obtain that there exists $C$ such that: 

$$
\sup_{t\in[0,T]}\int_{\T}|u(t,\cdot)|\le C
$$

which concludes the proof of the first estimate. Moreover  we have also proved that $y(S)\in L^{\infty}((0,T); L^{1}(\T))$. Therefore
$$
\sup_{t\in[0,T]}\int_{\T}|\partial_xr(t,\cdot)|\le C. 
$$

\end{proof}

This spatial regularity and the spatial regularity on $S$ can be translated to a spatial regularity on both $\rho$ and $\mu$ that are functions of $r,S$.  

\subsection{BV estimates on $\rho$ and $\mu$ and compactness}

Some of the following bounds are not suitably integrable in time. However, we prove below that this does not prevent strong compactness of the approximating scheme. 

\begin{proposition}[Spatial BV estimates]
\label{prop:BV_estim_rho_mu}
Under the hypotheses of Theorem~\ref{thm:main}, there exists an exponent $0<\varepsilon\le1$ such that
$$
\partial_x\rho,\quad \partial_x\mu \;\in\;L^\varepsilon((0,T);L^1(\T)).
$$
\end{proposition}

\begin{proof}
We only treat $\rho$, since the treatment of $\mu$ is identical.  Recall
$$
\rho
=\frac{Se^{r}}{1+e^{r}}.
$$
Differentiating in $x$ gives
$$
\partial_x\rho
=\frac{e^r}{1+e^r}\partial_xS
+
\frac{Se^r}{(1+e^r)^2}\partial_xr.
$$
Both factors $\frac{e^r}{1+e^r}$ and $\frac{e^r}{(1+e^r)^2}$ are bounded uniformly.

Concerning the first term on the right hand side: 
$$
|\partial_x S| =  C |\partial_x S^{\frac{\alpha}{2}} |S^{1-\frac{\alpha}{2}}
$$

Therefore with Proposition~\ref{prop:spatial_S} we deduce
$$
\partial_x S\in L^{1}((0,T)\times \T)
$$
as a product of two factors each in $L^{2}((0,T)\times \T)$.
With Proposition~\ref{prop:spatial_S} that yields $S\in L^{1-\alpha}((0,T); L^{\infty}(\T))$ and $\partial_xr \in L^{\infty}((0,T); L^{1}(\T))$ from the previous section we also deduce that there exists $0<\varepsilon\le 1$ such that 
$$
S|\partial_xr|\in L^{\varepsilon}((0,T); L^{1}(\T)).
$$
Indeed, from $\partial_x r\in L^\infty((0,T); L^{1}(\T))$ we just need $S\in  L^\varepsilon((0,T); L^{\infty}(\T))$, hence we can take $\varepsilon=\max\{1-\alpha,2\alpha-1\}$ (we observe that we have $\varepsilon\geq 1/3$ but this is not crucial).
We conclude $\partial_x \rho \in L^{\varepsilon}((0,T); L^{1}(\T))$ as claimed.
\end{proof}

We now prove compactness in time for $\rho$ and $\mu$:
\begin{proposition}
There exists $s>0$ such that
\begin{equation*}
    \partial_t\rho, \partial_t \mu \in L^{2}((0,T); H^{-s}(\T)). 
\end{equation*}
\end{proposition}

\begin{proof}
Once again we only prove the proposition for $\rho$. Proposition~\ref{prop:spatial_S} shows that $\sqrt{S}\partial_x f'(S)= C\partial_x S^{\alpha-\frac{1}{2}}\in L^{2}((0,T)\times \T)$ and therefore $\sqrt{\rho}\partial_x f'(S)\in L^{2}((0,T)\times \T)$. 
Since 
$$
\partial_t \rho = \p_x(\sqrt{\rho}\sqrt{\rho}\partial_xf'(S)) + \p_x(\rho\p_xV)
$$
with $\rho\in L^{\infty}((0,T); L^{1}(\T))$ we deduce that $\partial_t \rho \in L^{2}((0,T); Y)$ where $Y$ is the space of functions which are the derivatives of $L^1$ functions. This space can be continuously embed in an $H^{-s}$ space for some $s>0$.  
\end{proof}

In order to conclude the main theorem it remains to prove the following proposition. 

\begin{proposition}\label{prop:Aubin-Lions}
Let $\{\rho_n\}_n$ be a sequence of nonnegative functions. Let $0<\varepsilon<1 $. Assume there exists $C$ independent of $n$ such that
\begin{align*}
&\|\rho_n\|_{L^{\infty}((0,T); L^{1}(\T)) }\le C\\
& \|\partial_x\rho_n\|_{L^{\varepsilon}((0,T); L^{1}(\T))}\le C\\
&\|\partial_t \rho_n\|_{L^{2}((0,T); H^{-s}(\T))}\le C.
\end{align*}
Then up to a subsequence not relabed, there exists $\rho\in L^{\infty}((0,T); L^{1}(\T))$ with $\partial_t\rho\in L^{2}((0,T); H^{-s}(\T))$ such that $\rho_n\to \rho$ strongly in $L^{1}((0,T)\times \T)$ and almost everywhere in $(0,T)\times\T.$
\end{proposition}

\begin{proof}
The existence of $\rho$ and its regularity follows by weak star compactness with the Banach-Alaoglu theorem. Concerning strong compactness, we first prove equicontinuity in space. We let $M>0$, $f_n(t)=\|\partial_x\rho_n(t)\|_{L^{1}(\T)}$ and the sets 
$$
A_n = \left\{t\in(0,T),\; f_n(t)\le M\right\}, \quad B_n=(0,T)\setminus A_n.
$$
Thanks to our assumptions, there exists a constant $C$ independent of $n$ such that 
$$
|B_n|\le \frac{C}{M^{\varepsilon}}.
$$
Let $h>0$. We compute
\begin{align*}
\int_{0}^{T}\int_{\T}|\rho_n(t,x+h)-\rho_n(t,x)|&= \int_{A_n}\int_{\T}|\rho_n(t,x+h)-\rho_n(t,x)| + \int_{B_n}\int_{\T}|\rho_n(t,x+h)-\rho_n(t,x)|\\
&\le h \int_{A_n} \|\partial_x\rho_n\|_{L^{1}(\T)} + 2|B_n|\|\rho_n\|_{L^{\infty}((0,T);L^{1}(\T))}\\
&\le CMTh + \frac{C}{M^{\varepsilon}}.
\end{align*}
Choosing $M^{1+\varepsilon} = \frac{1}{h}$ we deduce that for all $n$ we have
\begin{align*}
\int_{0}^{T}\int_{\T}|\rho_n(t,x+h)-\rho_n(t,x)|\le Ch^{\frac{\varepsilon}{1+\varepsilon}}.
\end{align*}

Now we prove equicontinuity in time, and we consider a sequence of  mollifiers in space $\varphi_{\delta}(x)=\frac{1}{\delta}\varphi\left(\frac{x}{\delta}\right)$ for a smooth  nonnegative $\varphi$. Let $k>0$ and we compute
\begin{align*}
&\int_{0}^{T-k}\int_{\T}|\rho_n(t+k,x)-\rho_n(t,x)|\le \int_{0}^{T-k}\int_{\T}|\rho_n(t+k,\cdot)\ast\varphi_{\delta} -\rho_n(t,\cdot)\ast\varphi_\delta| \\ &+ \int_{0}^{T-k}\int_{\T}|\rho_n(t+k,\cdot)\ast\varphi_\delta-\rho_n(t+k,x)|+ \int_{0}^{T-k}\int_{\T}|\rho_n(t,\cdot)\ast\varphi_\delta-\rho_n(t,x)|
\end{align*}
By the previous equicontinuity in space, we can prove (it is a corollary of the Riesz-Frechet-Kolmogorov theorem) that the last two integrals are bounded by some monotone $\omega_1(\delta)$ independent of $n$ with $\omega_1(\delta)\to 0$ as $\delta\to 0$.    

Concerning the first term on the right-hand side we use the uniform  bound on $\partial_t\rho_n$ which yields 

$$
\int_{0}^{T-k}\int_{\T}|\rho_n(t+k,\cdot)\ast\varphi_{\delta} -\rho_n(t,\cdot)\ast\varphi_\delta|\le Ck\sqrt{T}\|\varphi_\delta\|_{H^{s}(\T)}\le \frac{Ck\sqrt{T}}{\delta^{\frac{1}{2}+s}}.
$$
Choosing $\delta^{\frac{1}{2}+s}=\sqrt{k}$ we find that there exists a monotone $\omega_{2}(k)$ such that $\omega_{2}(k)\to 0$ as $k\to 0$ such that
$$
\int_{0}^{T-k}\int_{\T}|\rho_n(t+k,x)-\rho_n(t,x)|\le \omega_{2}(k).
$$

Combining both equicontinuity and applying the Riesz-Fréchet-Kolmogorov theorem then yields the result. 
\end{proof}

With this proposition we are now able to pass to the limit and prove the main theorem. 

\begin{proof}[Proof of Theorem~\ref{thm:main}]
Let $\{\rho_n\}_n$,$\{\mu_n\}_n$ be a sequence of smooth functions constructed in an approximated scheme satisfying~\eqref{eq:rho}--\eqref{eq:mu} (up to a viscosity term) and with initial condition $\rho_{0},\mu_0$ respectively (up to smoothing also the intial condition). The difficulty is to pass to the limit in the term $\rho_n\partial_x f'(S_n)$, and same for $\mu_n$ (after integration by parts against a test function). Let $\rho$ and $\mu$ be the limit of $\{\rho_n\}_n$ and $\{\mu_n\}_n$ respectively (easily obtained from the previous sections), and $S$ their sum. 
We write this term as 
$$
\rho_n\partial_x f'(S_n)= \frac{\rho_n}{\sqrt{S_n}}\sqrt{S_n}\partial_x f'(S_n) =C\frac{\rho_n}{\sqrt{S_n}}\partial_x S_n^{\alpha-\frac{1}{2}}.
$$
We first show that $\frac{\rho_n}{\sqrt{S_n}}$ converges strongly to $\frac{\rho}{\sqrt{S}}$ in $L^2((0,T)\times \T)$. From Proposition~\ref{prop:Aubin-Lions} we know that it converge a.e. to $\frac{\rho}{\sqrt{S}}$ by continuity. Then 
$$
\left(\frac{\rho_n}{\sqrt{S_n}}\right)^2\le \rho_n,
$$
where $\rho_n$ converges strongly in $L^{1}((0,T)\times\T)$ from Proposition~\ref{prop:Aubin-Lions} and therefore is uniformly integrable. Thus by Vitali's convergence theorem (or the generalized Lebesgue dominated convergence theorem), $\frac{\rho_n}{\sqrt{S_n}}$ converges strongly in $L^2((0,T)\times \T)$ to $\frac{\rho}{\sqrt{S}}$. From Proposition~\ref{prop:spatial_S} we know that $\partial_x S_n^{\alpha-\frac{1}{2}}$ converges weakly in $L^{2}((0,T)\times \T)$. To identify it's limit it is enough to do it in the distributional sense, that is to identify the limit of $S_n^{\alpha-\frac{1}{2}}$. By convergence a.e. of $S_n$ and bounds on $S^{\alpha-\frac{1}{2}}$ in $L^{2}((0,T)\times \T)$ (since its gradient is bounded in this space) we deduce that $S_n^{\alpha-\frac{1}{2}}$ converges to $S^{\alpha-\frac{1}{2}}$ strongly in $L^{1}((0,T)\times \T)$ by Vitali's convergence theorem. This concludes the proof of the theorem. 
\end{proof}

\paragraph{Acknowledgements}
This work was supported by the European Union via the ERC
AdG 101054420 EYAWKAJKOS project.
The authors would like to thank Alpár Mészáros and Guy Parker for very useful discussions on the topic. In a previous draft of this paper we included an extra assumption $\alpha>2/3$ which we were able to remove thanks to these discussions. 
\bibliographystyle{siam}
\bibliography{biblio}

\end{document}